\documentclass[12pt,amscd]{amsart}
\usepackage{graphicx} 
\usepackage{amsmath,amsxtra,amssymb,latexsym, amscd,amsthm}

\usepackage{xcolor}

\DeclareMathOperator{\supp}{supp}

\DeclareMathOperator{\ass}{Ass}

\DeclareMathOperator{\astab}{astab}

\newtheorem{thm}{Theorem}[section]
\newtheorem{cor}[thm]{Corollary}
\newtheorem{lem}[thm]{Lemma}

\theoremstyle{definition}
\newtheorem{defn}[thm]{Definition}
\newtheorem{exm}[thm]{Example}

\def\P{{\mathcal P}}

\def\Z {\mathbb Z}

\def\p {\mathfrak p}
\def\q {\mathfrak q}
\def\m {\mathfrak m}

\def\ab {\mathbf a}
\def\bb {\mathbf b}

\title{Associated primes of powers of edge ideals of edge-weighted trees}

\author{Jiaxin Li}
\address{School of Mathematical Sciences, Soochow University, Suzhou, Jiangsu, 215006, P.R.~China}
\email{lijiaxinworking@163.com}

\author{Tran Nam Trung}
\address{Institute of Mathematics, Vietnam Academy of Science and Technology, 18 Hoang Quoc Viet, 10072 Hanoi, Vietnam}
\email{tntrung@math.ac.vn}

\author{Guangjun Zhu$^{\ast}$}
\address{School of Mathematical Sciences, Soochow University, Suzhou, Jiangsu, 215006, P.R.~China}
\email{zhuguangjun@suda.edu.cn}

\thanks{$^{\ast}$ Corresponding author}

 \date{\today}

\begin{document}



\begin{abstract} In this paper, we give a complete description of the associated primes of each power of the edge ideal of an increasing weighted tree. 
\end{abstract}
\maketitle

\section*{Introduction}

Let $R=K[x_1, \ldots, x_n]$ be a polynomial ring in $n$ variables over a field $K$. For an ideal $I\subset R$ and an integer $t\geqslant 1$, let $\ass(I^t)$ denote the set of  associated primes of $I^t$. Brodmann \cite{B} showed that $\ass(I^t)$ stabilizes  for all sufficiently large  $t$, meaning there exists a positive integer $t_0$ such that $\ass(I^t)=\ass(I^{t_0})$  for all $t\geqslant t_0$.  By virtue of this result, it is interesting to  describe  the set  $\ass(I^t)$ for each $t\geqslant 1$. This problem is difficult even when $I$ is a square-free monomial ideal (see \cite{FHV, HM, MV}).  When $I$ is an edge ideal, the associated primes of $I^t$ are first constructed algorithmically in \cite{CMS}, and then described completely in \cite{LT}.

Given a simple graph $G=(V,E)$ with the vertex set $V = \{x_1,\ldots,x_n\}$, recall that the ideal $I(G)$ of $R$ is generated by the monomials $x_ix_j$ where $x_ix_j$ is an edge of $G$. Then, every associated prime of $I(G)^t$ is of the form $(C)$, where $C$ is a vertex cover of $G$. In fact, for each vertex cover $C$ of $G$ and $t\geqslant 1$, Lam and Trung \cite{LT} gave a criterion for $(C)\in \ass(I(G)^t)$.

Now, moving away from square-free monomial ideals, we define  a weight function, $\omega \colon E \to \Z_{>0}$, on the edge set of $G$. The pair $(G,\omega)$  is called an {\it edge-weighted graph} (or simply a weighted graph), and is denoted by  $G_\omega$. The {\it weighted edge ideal} of $G_\omega$ is the monomial ideal of $R$ defined as follows (see \cite{PS}):
$$I(G_\omega) = ((x_ix_j)^{\omega(x_ix_j)}\mid x_i x_j\in E).$$

Since $\sqrt{I(G_\omega)^t} = I(G)$,  every associated prime of $I(G_\omega)^t$ is of the form $(C)$,  where $C$ is a vertex cover of $G$. Thus,  to describe the set $\ass(I(G_\omega)^t)$,  we  must determine if a vertex cover of $G$ forms an associated prime of $I(G_\omega)^t$. In this paper, we introduce the notion of {\it increasing weighted tree} and investigate this problem  when $G_\omega$ is such a weighted tree.

We say that $G_\omega$ is an increasing weighted tree if $G$  is a tree and there exists a vertex $v$, which is called a root of $G_\omega$, such that the weight function on every simple path from a leaf to root $v$ is increasing, i.e.,
 if
$$v_1\to v_2\to v_3\to \cdots \to v_k =v$$
is a  simple path from a leaf $v_1$ to $v$ of length at least $2$, then $\omega(v_iv_{i+1}) \leqslant \omega(v_{i+1}v_{i+2})$ for $i=1,\ldots, k-2$.

Let $C$ be a vertex cover  of $G$ such that $C\ne V$, and let $S = V\setminus C$. For each $u\in N_G(S)$, set $\nu_S(u) = \min\{\omega(zu)\mid z\in S \cap N_G(u)\}$. We say that $C$ is a {\it strong vertex cover} of $G_\omega$  if either $C$ is a minimal vertex cover of $G$ or, for every $w\in C \setminus N_G(S)$, there is a path from $w$ to a vertex $x$ in $N_G(S)$ as
$$w=w_1\to w_2\to \cdots  \to w_{k-1}\to w_k=x
$$
such that $\omega(w_{k-1}w_k) < \nu_S(x)$, but $w_1,\ldots,w_{k-1}\notin N_G(S)$.

If such a path also satisfies   $k\geqslant 3$ and $\omega(w_1w_2)=\omega(w_2w_3)$, then $w_2$ is called a {\it special} vertex of $C$. Let   $s(C)$ be the number of special vertices of $C$.

\medskip

Our main result is the following theorem.

\medskip

\noindent {{\bf Theorem \ref{main}}. \it Let $G_\omega$ be an increasing weighted tree and $t\geqslant 1$. If $C$ is a vertex cover of $G$, then $(C)$ is an associated prime of $I(G_\omega)^t$ if and only if $C$ is a strong vertex cover of $G_\omega$ and  $s(C)+1\leqslant t$.}

\medskip
For an ideal $I$, let $\astab(I)$ be the  smallest  positive integer $t_0$ such that $\ass(I^t)$ is constant for $t\geqslant t_0$. Let $\ass^\infty(I)$ denote the stable set $\ass(I^t)$ for $t\geqslant \astab(I)$. An immediate consequence of Theorem \ref{main} is that $$\ass^\infty(I(G_\omega)) = \{(C) \mid C \text{ is a strong vertex cover of } G_\omega\}.$$

Furthermore, we can provide precise formulas for both $\astab(I(G_\omega))$ and  for the index of stability of every associated prime of $\ass^\infty(I(G_\omega))$. It is worth mentioning that an upper bound for $\astab(I)$ is obtained in \cite{H} for every monomial ideal $I$, but this
bound is very large and not optimal. When $I$ is an edge ideal of a simple graph, a precise formula for $\astab(I)$ is provided in \cite{LT}. For the edge ideal of an increasing weighted tree $G_\omega$,  Theorem \ref{main} yields
$$\astab(I(G_\omega)) = \max\{s(C)+1\mid C \text{ is a strong vertex cover of } G_\omega\}.$$

The paper is organized as  follows:   Section \ref{sec:prelim} explores increasing weighted trees $G_\omega$.  In that section,  we characterize strong vertex covers of $G_\omega$ in terms of some weighted subgraphs and provide an efficient method for computing the number $s(C)$. Section \ref{sec:Second}  is devoted to proving the main result. The basic idea is the relationship between the associated primes of $I(G_\omega)^t$ and the strong vertex covers of $G_\omega$.

\section{Increasing weighted trees}
\label{sec:prelim}

In this section, we will explore increasing weighted trees. First, we will review some definitions and terminology from graph theory. Let $G$ be a graph. We often use  $V(G)$ and $E(G)$ to denote the vertex and the edge sets of $G$, respectively. If $u$ is a vertex in $G$, its neighborhood is the set  $N_G(u)=\{z \in V(G)\mid  zu\in E(G)\}$ and its degree, denoted by $\deg_{G}(u)$, is the size of $N_G(u)$.
If $\deg_G(u) = 1$, then $u$ is called a leaf. An edge that is incident with a leaf is called a pendant.  For any  $u\in V(G)$, let $L_G(u) = \{x\in N_G(u)\mid x \text{ is a leaf of } G\}$.

A subset $C \subseteq V(G)$ is a  {\it vertex cover}  of $G$ if every edge of $G$ has at least one endpoint in $C$. A minimal vertex cover  of $G$ is a vertex cover  of $G$ that is minimal with respect to inclusion.

The dual concept to the vertex cover is the {\it independent set}. Recall that an independent set of a graph $G$ is a collection of vertices with no two vertices adjacent to each other. Thus, the  complement of a vertex cover of $G$ in $V(G)$ is an independent set of $G$, and vice versa. Given  an independent set $S$ in $G$, its neighborhood is
$$N_G(S) = \{u\in V(G) \mid u \notin S \text{ and } N_G(u)\cap S \ne \emptyset\}.$$
We denote $G[S]$ to be the induced subgraph of $G$ on $S$, and  $G\setminus S$ to be  the induced subgraph of $G$ on $V (G)\setminus S$.

\medskip

We now define increasing paths in a weighted graph.

\begin{defn} Let $G_\omega$ be a weighted graph. Then,
\begin{enumerate}
    \item A simple path in $G$ is a sequence of distinct vertices: $v_1,v_2,\ldots, v_k$, where $v_iv_{i+1}\in E(G)$ for $i=1,\ldots,k-1$. In this case, the length of this path is $k-1$.
    \item We write $v_1\to v_2\to \cdots \to v_k$ to indicate the path $v_1,v_2,\ldots, v_k$ traveling from $v_1$ to $v_k$.
    \item A simple path $v_1\to v_2\to\cdots\to v_k$ is an increasing path if $\omega(v_iv_{i+1}) \leqslant \omega(v_{i+1}v_{i+2})$ for $i=1,\ldots,k-2$; and it is a strictly increasing path if $\omega(v_iv_{i+1}) < \omega(v_{i+1}v_{i+2})$ for $i=1,\ldots,k-2$.
\end{enumerate}
\end{defn}

\begin{defn} A weighted tree $G_\omega$ is called an increasing weighted tree, if there is a vertex $v$ such that  every simple path from a leaf to $v$ is increasing. In this case,  $v$ is called the root, and  $(G_\omega,v)$ is an increasing weighted tree, meaning $G_\omega$ is an increasing weighted tree with a root $v$.
\end{defn}

\begin{lem}\label{inc} If $(G_\omega,v)$ is an increasing weighted tree, then
\begin{enumerate}
    \item Every simple path to $v$ is increasing.
    \item There is no simple path in $G$ of the form
    $$v_1\to v_2\to\cdots \to v_{k-1}\to v_k, \text{where } k \geqslant 4$$
    such that $\omega(v_1v_2) > \omega(v_2v_3)$ and $\omega(v_{k-2}v_{k-1}) < \omega(v_{k-1}v_k)$.
\end{enumerate}
\end{lem}
\begin{proof} (1) Let $v_1\to v_2\to\cdots \to v$ be a simple path in $G$. If $v_1$ is a leaf, then the path is increasing by the definition. Otherwise, we can find a simple path from some leaf to $v_1$ in $G$, say $u_1\to\cdots\to u_j=v_1$. Then,
$$u_1\to\cdots\to u_j = v_1\to v_2\to\cdots\to v$$
is a simple path in $G$ from the leaf $u_1$ to $v$. Therefore, the path $v_1\to v_2\to\cdots\to v$ is increasing.

$(2)$ By the assumption, we can deduce that $v_3 \ne v$. Let $v_3=u_1\to u_2\to\cdots \to v$ be a simple path from $v_3$ to the root $v$. If $v_2\ne u_2$, then
$$v_1\to v_2\to v_3=u_1\to u_2\to \cdots \to v$$
is a simple path from $v_1$ to the root $v$.  This contradicts  Part $(1)$, since $\omega(v_1v_2) > \omega(v_2v_3)$. Therefore,  $v_2=u_2$. Thus,
$$v_k\to v_{k-1}\to \cdots \to v_3=u_1\to u_2\to \cdots \to v$$
is a simple path from $v_k$ to the root $v$, which contradicts Part $(1)$  because  $\omega(v_kv_{k-1}) > \omega(v_{k-1}v_{k-2})$. Therefore, $(2)$ follows.
\end{proof}

Let $G_\omega$ be a weighted graph. If $H$ is a subgraph of $G$, then $H_\omega$ is the weighted graph whose weight function is the restriction of  $\omega$ to the edge set of $H$. This means that the weight of an edge $e$ of $H$ is  $\omega(e)$ when $e$ is viewed as  an edge of $G$. We also say that $H_\omega$ is a weighted subgraph of $G_\omega$.

\begin{lem}\label{sub-tree} If $(G_\omega,v)$ is an increasing weighted tree, then every weighted subtree of $G_\omega$ is also an increasing weighted tree.
\end{lem}
\begin{proof} Let $T$ be a subtree of $G$. If $v$ is a vertex of $T$, then $T_\omega$ is an increasing weighted tree by Lemma \ref{inc}. If $v$ is not a vertex of $T$, then we can choose a simple path of the form $v_1\to v_2\to \cdots\to v$ such that $v_1$ is the only vertex of $T$  on this path.  In this case,  $T_\omega$ is an increasing weighted tree with  root $v_1$. Indeed, let $u_1\to u_2\to \cdots \to u_j=v_1$ be any simple path in $T$. Then, $u_1\to u_2\to \cdots \to u_j=v_1\to v_2\to \cdots\to v$ is a simple path in $G$. It follows that $u_1\to u_2\to \cdots \to u_j$ is an increasing path by Lemma \ref{inc}, as required.
\end{proof}

\medskip

\begin{lem}\label{path} Assume that $(G_\omega,v)$ is an increasing weighted tree. If $G$ is not a star  graph with a root $v$, then there is a longest path $v= v_0\to v_1\to\cdots\to v_k$ in $G$ from  $v$ such that
\begin{enumerate}
    \item $v_k$ is a leaf;
    \item if $u\in N_G(v_{k-2})$ is a non-leaf, then $\omega(v_{k-1}v_{k-2})\leqslant \omega(v_{k-2}u)$;
    \item $N_G(v_{k-1})$ has only one non-leaf $v_{k-2}$;
    \item $\omega(v_{k-1}u) \leqslant \omega(v_{k-1}v_{k-2})$ for all $u\in N_G(v_{k-1})$;
    \item $\omega(v_{k-1}v_k) \leqslant \omega(v_{k-1}u)$ for all $u\in N_G(v_{k-1})$.
\end{enumerate}
\end{lem}

\begin{proof}
Let $\mathcal{P}$ be the set of the longest paths in $G$ that start at $v$. Since $G$ is not a star  graph, every path in $\mathcal{P}$ has the same length, say $k$, at least $2$. Let 
$$P \colon v= v_0\to v_1\to \cdots \to v_{k-2}\to v_{k-1}\to v_k$$
be a path in $\mathcal{P}$ such that $\omega(v_{k-2}v_{k-1})$ is the smallest.  We will show that, after modifying the last vertex,  this path is the desired one.

$(1)$ If $v_k$ is not a leaf, then there is a $u\in N_G(v_k)\setminus \{v_{k-1}\}$. Therefore,  $v=v_0\to v_1\to\cdots\to v_k\to u$ is a simple path of length $k+1$, which is a contradiction. Thus, $v_k$ is a leaf.

$(2)$ Assume that there is a non-leaf $w\in N_G(v_{k-2})$ such that $\omega(v_{k-2}w)<\omega(v_{k-2}v_{k-1})$. Then, $w\notin \{v_{k-1},v_{k-3}\}$. Since $w$ is not a leaf,  it is adjacent to a vertex $u\ne v_{k-2}$. Therefore, $v=v_1\to v_2\to\cdots\to v_{k-2}\to w\to u$ is a simple path of length $k$, and  belongs to $\mathcal{P}$. However, $\omega(v_{k-2}w)<\omega(v_{k-2}v_{k-1})$, which contradicts the choice of $P$, and $(2)$ follows.

$(3)$ If $x\neq v_{k-2}$ is a non-leaf of $G$ that is adjacent to $v_{k-1}$, and  there is a $y \in N_G(x)$ that is different from $v_{k-1}$, then the simple path $v=v_0\to v_1\to\cdots\to v_{k-2}\to v_{k-1}\to x\to y$ has length $k+1$, a contradiction. Thus,  $N_G(v_{k-1})$ has only one non-leaf $v_{k-2}$.

$(4)$ Let $x\in N_G(v_{k-1})\setminus\{v_{k-2}\}$. Then,  by the condition $(3)$,  $x$ is a leaf. Since the path $x\to v_{k-1}\to v_{k-2}\to\cdots \to v_0=v$ is simple, we have $\omega(v_{k-1}x)\leqslant \omega(v_{k-1}v_{k-2})$.

$(5)$ Let $x\in N_G(v_{k-1})$ be a leaf such that $\omega(xv_{k-1})\leqslant \omega(uv_{k-1})$ for every $u\in N_G(v_{k-1})$. By replacing $P$ with  $v=v_0\to v_1\to\cdots\to v_{k-1}\to x$, we obtain a simple path that satisfies all conditions $(1)$-$(5)$, and the lemma follows.
\end{proof}

\medskip

\begin{defn} Let $G_\omega$ be a weighted tree and let $S$ be an independent set of $G$.
\begin{enumerate}
    \item For every $u\in N_G(S)$, set
  $$\nu_S(u) = \min\{\omega(uz)\mid z\in S \cap N_G(u)\}.$$
    \item Define $G_S$ to be the graph with the vertex set $V(G)\setminus S$ and the edge set obtained from the edge set of $G\setminus S$ by removing every edge $uz$ such that  $u\in N_G(S)$ and $\omega(uz)\geqslant \nu_S(u)$.
\end{enumerate}
\end{defn}

\begin{lem}\label{components} Let $(G_\omega,v)$ be an  increasing weighted tree and let $S$ be an independent set of $G$. Then 
\begin{enumerate}
    \item $N_G(S)$ is an independent set of $G_S$.
    \item If $T$ is any connected component of $G_S$, then
    $|V(T)\cap N_G(S)| \leqslant 1$. Moreover, if $V(T)\cap N_G(S) = \{u\}$, then $(T_\omega,u)$ is  an increasing weighted tree.
\end{enumerate}
\end{lem}

\begin{proof} $(1)$ Assume by contradiction that the set $N_G(S)$ is not an independent set of $G_S$, then  there is $uv\in E(G_S)$ with $u,v\in N_G(S)$. Let $x,y\in S$ such that $xu,yv\in E(G)$, and $\omega(xu) > \omega(uv)$ and $\omega(yv) > \omega(uv)$. Since $G_S$ is a subtree of $G$, we have  $uv\in E(G)$. Therefore, there is a simple path $x\to u\to v\to y$ with $\omega(xu) > \omega(uv)$ and $\omega(yv) > \omega(uv)$. This contradicts Lemma \ref{inc}.  Therefore,  $N_G(S)$ is an independent set of $G_S$.

$(2)$ Assume that $V(T)\cap N_G(S)\ne \emptyset$. Let $w$ be an element in this intersection. Now, assume that there is a $u\in V(T)\cap N_G(S)$ with $u\ne w$. Let $x\in N_G(u)\cap S$ and $y\in N_G(w)\cap S$. Clearly, $x\ne y$, since $G$ has no cycles. Now let
$$u=u_1\to u_2\to\cdots\to u_{k-1}\to u_k=w$$
be a simple  path in $T$ from $u$ to $w$. Then,
$$x\to u=u_1\to u_2\to\cdots\to u_{k-1}\to u_k=w\to y$$
is a simple path in $G$ with $\omega(xu) > \omega(uu_2)$ and $\omega(u_{k-1}w) < \omega(wy)$, which contradicts Lemma \ref{path}. Therefore, $V(T)\cap N_G(S)$ has just one element $w$.

We now show that $(T_\omega,w)$ is an increasing tree. If $w=v$, then $(T_\omega,w)$ is an increasing weighted tree by Lemma \ref{path}.

Assume that $w\ne v$. We first note that $v\notin V(T)$. Indeed, if $v\in V(T)$, then there is a simple path in $T$ from $w$ to $v$ in the form $w=v_1\to v_2\to \cdots \to v_i = v$. Then, there is a simple path  from some vertex $x\in N_G(v) \cap S$ to $v$   in the form
$x\to w=v_1\to v_2\to \cdots \to v_i =v$. Since $\omega(xw) > \omega(wv_2)$, this contradicts Lemma \ref{inc}.  Therefore,  $v\notin V(T)$.

Now, take a simple path $y= y_1\to y_2\to \cdots \to y_p =v$ in $G$ from a vertex $y\in V(T)$ to $v$ such that $y$ is the only vertex of $T$  on this path. We will show that $y =w$. Indeed, if $y\ne w$, then there is a simple path in $T$ from $w$ to $y$ of the form
$$w= a_1\to a_2\to\cdots \to a_q=y,$$
where $q\geqslant 2$. Then,
$$x\to w=a_1\to a_2\to\cdots \to a_q=y= y_1\to y_2\to \cdots \to y_p =v$$
is a simple path from a vertex $x\in S\cap N_G(w)$ to $v$ in $G$ and  $\omega(xw) > \omega(wa_2)$. This contradicts Lemma \ref{inc}. Therefore, $y=w$.

Finally, if $b_1\to b_2\to\cdots\to b_j=w$ is  any simple path in $T$, then
$$b_1\to b_2\to\cdots\to b_j=w=y_1\to y_2\to \cdots \to y_p =v$$
is a simple path from $b_1$ to $v$ in $G$. By Lemma \ref{inc},
$b_1\to b_2\to\cdots\to b_j=w$
is an increasing path. This shows that $(T_\omega,w)$ is an increasing weighted tree, completing the proof.
\end{proof}

As a consequence, we can use Lemma \ref{components} to determine  whether a vertex cover $C$ of an increasing weighted tree $G_\omega$ is strong and to compute $s(C)$.

\medskip

\begin{defn} Let $G_\omega$ be an increasing tree with a root $v$. A vertex $w$ of $G$ is called a special vertex of $(G_\omega,v)$ if there is a simple path in $G$ to $v$ of the form: $$u\to w\to x\to\cdots\to v$$ such that $\omega(uw)=\omega(wx)$.
We define $s(v,G_\omega)$  as the number of special vertices of $(G_\omega,v)$.
\end{defn}

\begin{lem}\label{strong-vc} Let $C$ be a vertex cover of an increasing weighted tree $G_\omega$ such that $C\ne V(G)$. Let $S = V(G)\setminus C$, and assume that $N_G(S) = \{r_1,\ldots,r_k\}$. For each $i=1,\ldots,k$, let $T^i$ be a connected component of $G_S$ such that $r_i\in V(T^i)$. Then, $C$ is a strong vertex cover of $G_\omega$ if and only if $G_S$ has exactly $k$ connected components $T^1,\ldots,T^k$ such that  $r_i\in V(T^i)$. Morevover, if $C$ is a strong vertex cover  of $G_\omega$, then  $(T^i_\omega, r_i)$ is an increasing weighted tree for $i=1,\ldots,k$, and
$$s(C) = \sum_{i=1}^k s(r_i, T^i_\omega).$$
\end{lem}
\begin{proof}
    For each $i$, let $T^i$ be the connected component of $G_S$ such that $r_i\in V(T_i)$. Then, by Lemma \ref{components}, $V(T^i)\cap V(T^j) = \emptyset$ if $i\ne j$.

      Now, suppose that $C$ is a strong vertex cover of $G_\omega$.  If $C$ is a minimal vertex cover of $G$, then $S$ is an independent set of $G$, it is trivial. Otherwise, for every vertex $x$ in $C\setminus N_G(S)$, there is a simple path from $x$ to some vertex $y$ in $N_G(S)$ in the form
    $$x = v_1\to v_2\to \cdots \to v_{k-1}\to v_k=y, $$
such that $\omega(v_{k-1}y) < \nu_S(y)$, and $v_1,\ldots, v_{k-1}\notin N_G(S)$. This is obviously a path in $G_S$, meaning that  $x$ is a vertex of some $T^j$ and  $y=r_j$. This  shows that $G_S$ has $k$ connected components $T^1,\ldots, T^k$.

Assume that  $G_S$ has exactly $k$ connected components $T^1,\ldots,T^k$ such that  $r_i\in V(T^i)$. We will prove that $C$ is a strong vertex cover of $G_\omega$.
  If  $C$ is a minimal vertex cover of $G$, then the result is trivial. Otherwise, for any vertex $x\in C\setminus N_G(S)$, $x$ is a vertex of some $T^i$, since $G_S$ has exactly $k$ connected components $T^1,\ldots,T^k$ such that  $r_i\in V(T^i)$ and  $C\setminus N_G(S)\subseteq V(G_S)$. Therefore, there is a simple path from $x$ to $r_i$ in the form
    $$x = v_1\to v_2\to \cdots \to v_{k-1}\to v_k=r_i$$
such that $v_1,\ldots, v_{k-1}\notin N_G(S)$, and  $\omega(v_{k-1}r_i) < \nu_S(r_i)$ by the definition of $G_S$. Therefore, $C$ is a strong vertex cover of $G_\omega$.

    Finally, if $C$ is a strong vertex cover of $G_\omega$, then each  $T^i$ is a connected component of $G_S$ and $V(T^i)\cap N_G(S)=\{r_i\}$.  By Lemma \ref{components}(2),
    $(T^i_\omega, r_i)$ is an increasing weighted tree. We can   directly verify   that, for every $i$ and  every  vertex $x\in (C\setminus N_G(S))\cap V(T^i)$, $x$ is a special vertex of $(T^i_\omega,r_i)$ if and only if $x$ is special of $C$. Therefore,
    $$s(C) = \sum_{i=1}^k s(r_i, T^i_\omega)$$
    and the lemma follows.
\end{proof}

\section{Associated primes}
\label{sec:Second}

In this section, we will find the associated primes of $I(G_\omega)^t$, where $G_\omega$ is an increasing weighted tree. Throughout this section,  we will  assume that $V(G)=\{x_1,\ldots,x_n\}$ and that $\m=(x_1,\ldots,x_n)$ is the homogeneous maximal ideal of $R=K[x_1,\ldots,x_n]$.

For a monomial ideal $I\subseteq R$,  let  $\mathcal{G}(I)$ denote the unique minimal set of its monomial generators. For a positive integer $n$, the notation $[n]$ denotes the set $\{1,2,\dots,n\}$.

We need the following lemma.

\begin{lem}\label{dec-power} Let $I$ be a monomial ideal and let $x^py^q$ be a monomial in $\mathcal{G}(I)$, where $p \geqslant 1$ and $q\geqslant 1$, and $x$ and $y$ are variables. For any   $f$ in $\mathcal{G}(I)$ that satisfies 
\begin{enumerate}
    \item if  $f\ne x^py^q$, then $y\nmid f$,
    \item if  $x\mid f$, then $\deg_x(f)\geqslant p$.
\end{enumerate}
Then $(I^t \colon x^py^q)= I^{t-1}$ for all $t\geqslant 2$.
\end{lem}
\begin{proof} First,  $I^{t-1}\subseteq (I^t \colon x^py^q)$ is clear. Let $g\in (I^t\colon x^py^q)$ be a monomial, then $gx^py^q =hf_1\cdots f_t$, where $h$ is a monomial  and $f_1,\ldots,f_t\in \mathcal{G}(I)$.
If $y\mid f_j$ for some $j\in [t]$, then, by the  assumption (1), $f_j = x^py^q$. Therefore,  $g\in I^{t-1}$. If $y\nmid f_j$ for each $j\in [t]$, then by the expression of $gx^py^q$, we can deduce that $y^q\mid h$. If $x\nmid f_j$ for any $j\in [t]$, then,  $x^p\mid h$. Thus $g\in I^t$.
If $x\mid f_j$ for some $j\in [t]$, then, by the  assumption  (2), $x^p \mid f_j$. Thus $x^py^q \mid hf_j$. Therefore, $g\in I^{t-1}$.
We   complete the proof.
\end{proof}

 For any  $\mathbf{a}= (a_1, \ldots, a_n) \in \mathbb{N}^n$,  define the monomial $x^{\mathbf{a}} := \prod\limits_{i=1}^{n} x_i^{a_i}$ and write $\deg_{x_i}(x^{\mathbf{a}}) = a_i$ for each $i \in [n]$.
 
 \begin{lem}\label{maxIdeal} If $G_\omega$ is an increasing weighted tree, then $\m \notin \ass (I(G_\omega)^t)$ for all $t\geqslant 1$.
\end{lem}
 \begin{proof}
	   	  Let $I=I(G_\omega )$. We will prove the statement by induction on $n=|V(G)|$. If $n=2$, then $I=((x_1x_2)^{\omega (x_1x_2)})$. It is clear that $\m\notin \operatorname{Ass}(I^t)$ for all $t\ge 1$. Now, we assume that $n\ge 3$. Suppose, by contradiction, that $\m\in \operatorname{Ass}(I^t)$ for some $t\ge 1$, and let $t_0$ be the smallest such integer.  Then 
there exists a monomial $x^{\mathbf{a}}\notin I^{t_0}$ such that $\m=(I^{t_0}:x^{\ab})$. Choose a leaf $x_i$  such that $\omega (x_ix_j)=\min\{\omega (e)\mid e\in E(G_\omega )\}$ where $N_{G}(x_i)=\{x_j\}$, then 
	   	\[
  x_ix^{\ab}=hu_1\cdots u_{t_0},
	   	  \]
	   	  where $h$ is a monomial and each $u_{\ell}\in \mathcal{G}(I)$. Note that $x_i\nmid h$, since $x^{\ab}\notin I^{t_0}$, thus $x_i|u_{\ell}$ for some $\ell\in[t_0]$. Since $N_{G}(x_i)=\{x_j\}$, we have $u_{\ell}=(x_ix_j)^{\omega (x_ix_j)}$. Therefore,  $(x_ix_j)^{\omega (x_ix_j)}|x_ix^{\mathbf{a}}$,  which implies $a_i\ge \omega (x_ix_j)-1$ and $a_j\ge \omega (x_ix_j)$. If  $a_i\ge \omega (x_ix_j)$, then $(x_ix_j)^{\omega (x_ix_j)}|x^{\ab}$. Thus  $x^{\ab}=(x_ix_j)^{\omega (x_ix_j)}u$ for some monomial $u$. By Lemma \ref{dec-power}, we obtain that 
	   	  \[
	   	  \m=(I^{t_0}:x^{\ab})=((I^{t_0}:(x_ix_j)^{\omega (x_ix_j)}):u)=(I^{t_0-1}:u).
	   	  \]
	Therefore, $\m \in \operatorname{Ass}(I^{t_0-1})$,  contradicting the minimality of $t_0$. Therefore, $a_i=\omega (x_ix_j)-1$. 

 For each $k\neq i$,  note that $(x_ix_j)^{\omega (x_ix_j)}\nmid x_kx^{\ab}$, and thus
 $x_kx^{\ab}\in I((G \setminus x_i)_{\omega})^{t_0}$.
Also, note that $x_kx^{\bb}\in I((G \setminus x_i)_{\omega})^{t_0}$, where $\bb=(a_1,\ldots,a_{i-1},0,a_{i+1},\ldots,a_n)$. If $x^{\bb}\in I((G \setminus x_i)_{\omega})^{t_0}$, then $x^{\ab}=x_i^{a_i}x^{\bb}\in I^{t_0}$,
contradicting $x^{\ab}\notin I^{t_0}$. Therefore, $x^{\bb}\notin I((G \setminus x_i)_{\omega})^{t_0}$. Therefore, $(x_1,\ldots,x_{i-1},x_{i+1},\ldots,x_n)\in \operatorname{Ass}(I((G \setminus x_i)_{\omega})^{t_0})$.
Since $|V(G_\omega \setminus x_i)| = n-1$, by the inductive hypothesis,   $(x_1,\ldots,x_{i-1},x_{i+1},\ldots,x_n)\notin \operatorname{Ass}(I((G \setminus x_i)_{\omega})^t)$ for all $t\ge 1$, which is a contradiction.
Therefore,  $m\notin \operatorname{Ass}(I^t)$ for all $t\ge 1$.
 \end{proof}

The following example shows that Lemma \ref{maxIdeal} is no longer true if $G_\omega$ is not an increasing weighted tree.

\begin{exm}\label{maxIdealIn} Let $G$ be a path of length $4$ with the vertex set $V=\{x_i\mid i\in [5]\}$ and the edge set $E=\{x_1x_2,x_2x_3,x_3x_4,x_4x_5\}$. Define the weight function $\omega$ on $E$ by:
$$\omega(x_1x_2)=3,\omega(x_2x_3)=\omega(x_3x_4)=2,\omega(x_4x_5)=3.$$
Using  Macaulay2, we can verify that $(x_1,x_2,x_3,x_4,x_5)\in \ass(I(G_\omega)^5)$.
\end{exm}

 For any $x\in V(G_\omega)$,  we define $\mu(x) = \max\{\omega(xy)\mid y\in N_G(x)\}$.

\begin{lem} \label{inc-r-assoc} Let $(G_\omega,v)$ be an increasing weighted tree and let $m\geqslant\mu(v)$. Then  $\ass ((v^m,I(G_\omega))^t) \subseteq \ass ((v^m,I(G_\omega))^{t+1})$ for all $t\geqslant 1$.
\end{lem}
\begin{proof} Since $(G_\omega,v)$ is an increasing weighted tree,  there exists a leaf $y$  such that $y\ne v$ and $\omega (xy)=\min\{\omega (e)\mid e\in E(G_\omega )\}$, where $N_{G}(y)=\{x\}$.

For any $t\geqslant 1$,  by using Lemma \ref{dec-power}, we have $((v^m,I(G_\omega))^{t+1} \colon (xy)^{\omega(xy)}) = (v^m,I(G_\omega))^t$.
 According to  \cite[Lemma 3.3]{HM}, $\ass ((v^m,I(G_\omega))^{t})\subseteq \ass ((v^m,I(G_\omega))^{t+1})$,  as required.
\end{proof}

\begin{lem} \label{star} Let $(G_\omega,v)$ be a star graph with a root $v$ and let $m>\mu(v)$. Then  $\m\in \ass (v^m,I(G_\omega))$.
\end{lem}
\begin{proof} Let $I=I(G_\omega)$ and $f = v^{m-1} \prod\limits_{x\in V(G) \colon x\ne v} x^{\omega(xv)-1}$.
Then $f\notin (v^m,I)$.  Indeed, since $\deg_v(f) = m-1 $,  $v^m\nmid f$. Note that, for all $x\neq v$, $\deg_x(f) = \omega(xv)-1 $, thus  $(xv)^{\omega(xv)}\nmid f$.  Therefore, $f\notin I$. 

On the other hand, for each $u\in V(G)$, if $u= v$, then  $uf\in (v^m,I)$. Otherwise,  $uf= u^{\omega(uv)}v^{m-1} \prod\limits_{x\in V(G) \colon  x\notin \{u,v\}} x^{\omega(xv)-1}=(uv)^{\omega(uv)}\left(v^{m-\omega(uv)-1} \prod\limits_{x\in V(G) \colon  x\notin \{u,v\}} x^{\omega(xv)-1}\right)\in I$, and therefore $uf\in (v^m,I)$, and $\m\in \ass((v^m,I))$.
\end{proof}

\begin{lem} \label{sG1} Let $(G_\omega,v)$ be an increasing weighted tree and let $m > \mu(v)$. Then  $\m\in \ass((v^m,I(G_\omega))^t)$ for all $t\geqslant s(v,G_\omega)$+1.
\end{lem}
\begin{proof}  Let $I = I(G_\omega)$. We will prove the statement by induction on $n=|V(G)|$.  If  $G$ is a star graph with a root  $v$, then  the result follows from Lemmas \ref{inc-r-assoc} and  \ref{star}.

Assume that $n>2$ and $G$ is not a star graph. By Lemma \ref{inc-r-assoc}, it suffices to show that $\m\in \ass ((v^m,I(G_\omega))^{t_0})$, where $t_0=s(v,G_\omega)+1$.
For any  $x\in V(G)$, let $L_G(x) = \{u\in N_G(x)\mid \deg_G(u)=1\}$.
We consider the following two cases.

{\it Case $1$}: $G$ has a pendant edge $xy$ with $\deg_G(y)=1$, satisfying the following four conditions:
\begin{enumerate}
    \item $y\ne v$,
    \item $\omega(xy) \leqslant \omega(xz)$ for all $z\in N_G(x)$,
    \item $s(v,G_\omega)=s(v,G'_{\omega})$, where $G'=G\setminus y$, and
    \item either  there exists an $r\in L_G(x)\setminus\{y,v\}$, or $\omega(xy) < \omega(xz)$ for all $z\in N_G(x)\setminus\{y\}$.
\end{enumerate}
In this case, let $I'=I(G'_\omega)$ and $\m' = (z\mid z\ne y)$. Since $(G'_\omega,v)$ is an increasing weighted tree, the induction hypothesis implies that $\m'\in  \ass((v^m,I')^{t_0})$ by the condition (3).
Therefore,   there is a monomial $f\notin (v^m,I')^{t_0}$ and  $y\nmid f$ such that $\m' = ((v^m,I')^{t_0}\colon f)$.
Let $g = fy^{\omega(xy)-1}$, then $\deg_y(g)=\omega(xy)-1$. Thus, $g\notin (v^m,I)^{t_0}$ since $\m' = ((v^m,I')^{t_0}\colon f)$. Now, we will prove that $\m = ((v^m,I)^{t_0}\colon g)$. 

For any $z\ne y$, since $\m'\in  \ass((v^m,I')^{t_0})$,
 we have  $fz\in (v^m,I')^{t_0}$. Therefore,  $gz = (fz)y^{\omega(xy)-1}\in (v^m,I')^{t_0}\subseteq (v^m,I)^{t_0}$. This implies that $z\in  ((v^m, I)^{t_0} : g)$.

Next, we will show that $gy \in (v^m,I)^{t_0}$.  To do so, it is sufficient to  show that $f$ can be written as  $f = x^{\omega(xy)}f'$ where $f'$ is a monomial in $(v^m,I')^{t_0-1}$. We consider the following two subcases:

(i) If there exists $r\in L_G(x)\setminus\{y,v\}$, then $fr\in (v^m, I')^{t_0}$, since $\m' = ((v^m,I')^{t_0}\colon f)$.
We can write $fr$ as $fr = \gamma f_1f_2\cdots f_{t_0}$ where $\gamma$ is a monomial and $f_1,\ldots,f_{t_0}\in \mathcal{G}((v^m,I'))$. Note that since $f\notin (v^m,I')^{t_0}$, it is easy to see  that $r|f_j$ for some $j\in [t_0]$.
Without loss of generality, we can assume that $j=t_0$. By the choice of $r$,  $f_{t_0} = (xr)^{\omega(xr)}$. By the condition $(2)$, we have that  $\omega(xy) \leqslant \omega(xr)$. Therefore, $f = x^{\omega(xy)}f'$ and
$f' = \gamma x^{\omega(xr)-\omega(xy)}r^{\omega(xr)} f_1\cdots f_{t_0-1}\in (v^m,I')^{t_0-1}$.

(ii) If $\omega(xy) < \omega(xz)$ for all $z\in N_G(x)\setminus \{y\}$, then $xf\in  (v^m, I')^{t_0}$ since $\m' = ((v^m,I')^{t_0}\colon f)$. We can write $xf$ as $xf= \gamma' f'_1f'_2\cdots f'_{t_0}$ where $\gamma'$ is a monomial and $f'_1,\ldots,f'_{t_0}\in \mathcal{G}((v^m,I'))$. It is easy to see that  $x| f'_j$ for some $j\in [t_0]$, We  can  also assume that $j=t_0$, so $f'_{t_0} = (xz)^{\omega(xz)}$ for some $z\in N_G(x)\setminus \{y\}$, or $f'_{t_0} = x^m$ (this  case can be  occur if $x=v$). In both cases,   $\deg_x(f'_{t_0})\geqslant \omega(xy)+1$, Thus,  $f'_{t_0}$ can be written as  $f'_{t_0} = hx^{\omega(xy)+1}$, where $h$ is a monomial. Therefore, $f = x^{\omega(xy)}f'$ and $f' = \gamma' h x f'_1\cdots f'_{t_0-1}\in (v^m,I')^{t_0-1}$.

In both subcases, we have $g y = f' (xy)^{\omega(xy)} \in (v^m, I)^{t_0}$,  implying that $y\in  ((v^m, I)^{t_0} : g)$.
Therefore, $\mathfrak{m} \in \ass((v^m, I)^{t_0})$,  so the statement holds.

{\it Case $2$}: Assume that no pendant of $G$  satisfies Case $1$. By Lemma \ref{path}, there is a longest path $\P: v= v_0\to v_1\to\cdots\to v_{s-1}\to v_s$ in $G$ from the root $v$ such that
\begin{itemize}
	\item[(5)] $s \geqslant 2$;
	\item[(6)] $v_s$ is a leaf;
	\item[(7)] if $u\in N_G(v_{s-2})\setminus L_G(v_{s-2})$, then $\omega(v_{s-1}v_{s-2})\leqslant \omega(v_{s-2}u)$;
	\item[(8)] $\omega(v_{s-1}v_s) \leqslant \omega(v_{s-1}z)$ for all $z\in N_G(v_{s-1})$;
	\item[(9)] $N_G(v_{s-1})$ has only one non-leaf $v_{s-2}$;
	\item[(10)] $\omega(v_{s-1}z) \leqslant \omega(v_{s-1}v_{s-2})$ for all $z\in N_G(v_{s-1})$.
\end{itemize}
Note that $v \notin L_G(v_{s-1}) \cup \{v_{s-1}\}$ and the pendant  $v_{s-1}v_s$ satisfies the  conditions (1) and (2). In this case, we first prove that condition (3) is equivalent to condition (4).

(3)$\implies$(4): If  $L_G(v_{s-1})=\{v_s\}$, then  $N_G(v_{s-1})\setminus \{v_s\}=\{v_{s-2}\}$  by the condition (5) and (9). By  the  condition (3),  $\omega(v_{s-1}v_s) < \omega(v_{s-2}v_{s-1})$.

(4)$\implies$(3): Let   $G''=G\setminus v_s$.  If $\omega(v_{s-1}v_s) < \omega(v_{s-1}z)$ for all $z\in N_G(v_{s-1})\setminus\{v_s\}$, then $\omega(v_{s-1}v_s) < \omega(v_{s-2}v_{s-1})$,  implying that $s(v,G''_{\omega})=s(v,G_\omega)$. Otherwise, there exists a $z\in N_G(v_{s-1})\setminus\{v_s\}$ such that $\omega(v_{s-1}v_s) \geqslant \omega(v_{s-1}z)$. Thus, again, using the condition (4),
$L_G(v_{s-1})\setminus\{v_s,v\}\neq \emptyset$.
Using the conditions (8) and (10), we can deduce that $\omega(v_{s-1}v_s)=\omega(v_{s-1}z)$ for some $z\in L_G(v_{s-1})\setminus\{v_s,v\}$. Therefore, $s(v,G''_{\omega})=s(v,G_\omega)$.

Below, we only consider cases where the pendant $v_{s-1}v_s$  does not satisfy conditions (4). That is,  $L_G(v_{s-1})=\{v_s\}$, since $s \geqslant 2$, and there is a $z\in N_G(v_{s-1})\setminus\{v_s\}$  such that $\omega(v_{s-1}v_s)\geqslant \omega(v_{s-1}z)$.  By the condition (9), $N_G(v_{s-1})=\{v_{s-2},v_s\}$. Thus $\omega(v_{s-1}v_{s})=\omega(v_{s-1}v_{s-2})$ by the condition (8).
Therefore,  $s(v,G_{\omega})=s(v,G'_\omega)+1$. 

First, we will show that, for the  longest path $\P$,   $\omega(v_{s-1}v_{s-2}) \leqslant \omega(zv_{s-2})$  for all  $z\in N_G(v_{s-2})$.

 The case $L_G(v_{s-2})=\emptyset$ follows from  the condition $(7)$.
Now, assume that $L_G(v_{s-2})\neq \emptyset$.   Using the condition $(7)$ again,  it suffices to show  that $\omega(v_{s-1}v_{s-2}) \leqslant \omega(zv_{s-2})$ 
for  all  $z\in L_G(v_{s-2})$. Suppose for contradiction that there is an  $\alpha \in L_G(v_{s-2})$ such that $\omega(v_{s-2}v_{s-1})> \omega(v_{s-2}\alpha)$. Moreover, we
can assume that  $\omega(zv_{s-2})\geqslant \omega(v_{s-2}\alpha)$ for all $z\in L_G(v_{s-2})$.
Then, by the condition  (7), we have that $\omega(v_{s-2}u)>\omega(v_{s-2}\alpha)$ for all $u\in N_G(v_{s-2})\setminus L_G(v_{s-2})$. This implies that 
$s(v,G_\omega)=s(v,(G\setminus \alpha)_{\omega})$. Therefore,  the pendant $wv_{s-2}$ satisfies  the  four conditions of Case $1$, which is a contradiction.

Next, let  $I''=I(G''_\omega)$ and $\m'' = (z\mid z\ne v_{s})$. Since $(G''_\omega,v)$ is an increasing weighted tree,  by  the induction hypothesis, $\m'' \in \ass((v^m,I'')^{t_0-1})$,  since $s(v,G''_\omega)=s(v,G_{\omega})-1$. Therefore, there is a monomial  $f_1$ such that  $v_{s}\nmid f_1$ and $\m'' = ((v^m,I'')^{t_0-1}\colon f_1)$. Let  $g_1= f_1(v_{s-2}v_{s-1})^{\omega(v_{s-2}v_{s-1})}v_{s}^{\omega(v_{s-1}v_{s})-1}$. We  will prove that  $g_1\notin (v^m,I)^{t_0}$ and that $\m = ((v^m,I)^{t_0}\colon g_1)$.

If  $g_1\in (v^m,I)^{t_0}$, then  $(v_{s-1}v_{s})^{\omega(v_{s-1}v_{s})}\nmid g_1$, since  $\deg_{v_{s}}(g_1) = \omega(v_{s-1}v_{s})-1$. This implies that  $g_1 \in (v^m,I'')^{t_0}$.
By the expression of $g_1$, we have $f_1(v_{s-2}v_{s-1})^{\omega(v_{s-2}v_{s-1})} \in (v^m,I'')^{t_0}$.
Therefore, \[
 f_1\in ((v^m,I'')^{t_0} \colon (v_{s-2}v_{s-1})^{\omega(v_{s-2}v_{s-1})}) = (v^m,I'')^{t_0-1},
 \]
where the above equality holds because of the fact that $\omega(v_{s-2}v_{s-1}) \leqslant \omega(v_{s-2} z)$  for all  $z\in N_G(v_{s-2})$ and Lemma \ref{dec-power}. 
This contradicts the fact that  $f_1\notin  (v^m,I'')^{t_0-1}$. Therefore,  $g_1\notin (v^m,I)^{t_0}$.

For any $\beta \in V(G)$, if $\beta \ne v_{s}$, then
$\beta g_1 =[(\beta f_1)(v_{s-2}v_{s-1})^{\omega(v_{s-2}v_{s-1})}]v_s^{\omega(v_{s-1}v_{s})-1} \in (v^m,I)^{t_0}$
since  $\m'' = ((v^m,I'')^{t_0-1}\colon f_1)$. Otherwise, we have 
\begin{align*}
	v_{s}g_1& =f_1(v_{s-2}v_{s-1})^{\omega(v_{s-2}v_{s-1})}v_{s}^{\omega(v_{s-1}v_{s})}\\
	&= [(f_1v_{s-2})(v_{s-1}v_{s})^{\omega(v_{s-1}v_{s})}]v_{s-1}^{\omega(v_{s-2}v_{s-1})-\omega(v_{s-1}v_{s})}v_{s-2}^{\omega(v_{s-2}v_{s-1})-1}  \in (v^m,I)^{t_0}.
\end{align*}
Therefore, $\m = ((v^m,I)^{t_0}\colon g_1)$ and  $\m \in \ass((v^m,I)^{t_0})$. We  have completed  the proof.
\end{proof}

\begin{lem} \label{sG2} Let $(G_\omega,v)$ be an increasing weighted tree and let $m > \mu(v)$. Then $\m\in \ass((v^m,I(G_\omega))^t)$ if and only if $t\geqslant s(v,G_\omega)$+1.
\end{lem}
\begin{proof} Let $I=I(G_\omega)$. By Lemma \ref{sG1}, it suffices to show that if $\m\in\ass((v^m,I)^t)$, then $t\geqslant s(v,G_{\omega})+1$. 
 We now prove this assertion by induction on $n=|V(G)|$. 
 If  $G$ is a star graph with a root  $v$, then  $s(v,G_\omega)=0$ and the assertion follows from Lemmas \ref{inc-r-assoc} and  \ref{star}.
 
 Assume that $n > 2$ and $G$ is not a star graph. Let $k=\min\{\ell\mid \m\in \ass((v^m,I)^{\ell})\}$ and let $f$ be a monomial in $R$ such that $\m = ((v^m,I)^k \colon f)$. 
We will prove that $k\geqslant s(v,G_\omega)+1$.

By Lemma \ref{path},  there exists a longest path $v= v_0\to v_1\to\cdots\to v_{s-1}\to v_s$ in $G$ from the root $v$ such that 
\begin{enumerate}
	\item $s\geqslant 2$;
	\item $v_s$ is a leaf;
	\item if $z\in N_G(v_{s-2})$ is a non-leaf, then $\omega(v_{s-1}v_{s-2})\leqslant \omega(v_{s-2}z)$;
	\item $N_G(v_{s-1})$ has only one non-leaf $v_{s-2}$;
	\item $\omega(v_{s-1}z) \leqslant \omega(v_{s-1}v_{s-2})$ for all $z\in N_G(v_{s-1})$;
	\item $\omega(v_{s-1}v_s) \leqslant \omega(v_{s-1}z)$ for all $z\in N_G(v_{s-1})$.
\end{enumerate}
First, we will prove the following  three  claims:

{\it Claim $1$}: $(v_{s-1}v_{s})^{\omega(v_{s-1}v_{s})} \nmid f$. 

If $(v_{s-1}v_{s})^{\omega(v_{s-1}v_{s})} \mid f$, then $f = g(v_{s-1}v_{s})^{\omega(v_{s-1}v_{s})}$, where $g$ is a monomial. Together with the condition  $(6)$ and  Lemma \ref{dec-power},   this yields
\[
((v^m,I)^k \colon (v_{s-1}v_{s})^{\omega(v_{s-1}v_{s})}) = (v^m,I)^{k-1}.
\]
Therefore,
$$\m = ((v^m,I)^k \colon f) = (((v^m,I)^k\colon (v_{s-1}v_{s})^{\omega(v_{s-1}v_{s})})\colon g) = ((v^m,I)^{k-1}\colon g).$$
Hence $\m\in \ass((v^m,I)^{k-1})$. This contradicts the minimality of $k$, so  $(v_{s-1}v_{s})^{\omega(v_{s-1}v_{s})}\nmid f$, as claimed.

\medskip

{\it Claim $2$}: $\deg_{v_s}(f) = \omega(v_{s-1}v_{s})-1$ and $\deg_{v_{s-1}}(f)\geqslant \omega(v_{s-1}v_{s})$. 

Note that  $v_sf\in (v^m,I)^k$, we can write $v_{s}f$ as $v_{s}f= hf_1\cdots f_k$, where $h$ is a monomial and $f_1,\ldots,f_k\in \mathcal{G}((v^m,I))$. Since  $f\notin (v^m,I)^k$,    $v_{s}\mid f_j$ for some $j\in [k]$. Therefore,  $f_j=(v_{s-1}v_{s})^{\omega(v_{s-1}v_{s})}$, since  $v_{s}$ is a leaf of $G$.
In particular, $\deg_{v_{s-1}}(f) \geqslant \omega(v_{s-1}v_{s})$ and $\deg_{v_{s}}(f) \geqslant \omega(v_{s-1}v_{s})-1$.  By Claim $1$, $(v_{s-1}v_{s})^{\omega(v_{s-1}v_{s})}\nmid f$, which forces $\deg_{v_{s}}(f) < \omega(v_{s-1}v_{s})$, and thus $\deg_{v_{s}}(f) = \omega(v_{s-1}v_{s})-1$, as claimed.

\medskip
{\it Claim $3$}: If $s(v,G'_\omega)=s(v,G_\omega)$,  where $G'_{\omega}=G_\omega\setminus v_s$, then $k\ge s(v,G_\omega)+1$. 

Let $\m'=(z\mid z\neq v_s)$. For any $z\in \m'$,  $fz\in (v^m,I)^k$ since $\m = ((v^m,I)^k \colon f)$. 
We can write $fz$ as 
\[
fz = \gamma g_1\cdots g_k,
\]
where $\gamma$ is a monomial and $g_1,\ldots,g_k\in \mathcal{G}((v^m,I))$. Since  $z\ne v_{s}$ and by  Claim $2$,  $\deg_{v_{s}}(fz)=\omega(v_{s-1}v_{s})-1$. Therefore,  $g_i\ne (v_{s-1}v_{s})^{\omega(v_{s-1}v_{s})}$ for all $i\in [k]$. In particular, $fz\in (v^m,I')^k$, which implies that  $\m'=((v^m,I')^k\colon f)$. Therefore, $\m'\in \ass((v^m,I')^k)$. Since $|V(G')| = n-1$, by the induction hypothesis, $k\geqslant s(v,G'_{\omega})+1=s(v,G_{\omega})+1$.

\medskip
 We will prove that $k\geqslant s(v,G_{\omega})+1$ by considering the following five cases.
\begin{itemize}
     \item[(i)]  $\omega(v_{s-1}v_{s}) < \omega(v_{s-2}v_{s-1})$;
      \item[(ii)]  $\omega(v_{s-1}v_{s}) = \omega(v_{s-2}v_{s-1})$ and $L_G(v_{s-1})\setminus \{v_s\}\ne \emptyset$;
    \item[(iii)] $\omega(v_{s-1}v_{s})=\omega(v_{s-2}v_{s-1})$,  $L_G(v_{s-1})=\{v_s\}$ and $L_G(v_{s-2})= \emptyset$;
    \item[(iv)] $\omega(v_{s-1}v_{s}) = \omega(v_{s-2}v_{s-1})$,  $L_G(v_{s-1})=\{v_s\}$,  $L_G(v_{s-2})\ne \emptyset$ and $\omega(v_{s-2}u) < \omega(v_{s-2}v_{s-1})$ for some $u\in L_G(v_{s-2})$; 
    \item[(v)] $\omega(v_{s-1}v_{s}) = \omega(v_{s-2}v_{s-1})$,  $L_G(v_{s-1})=\{v_s\}$,  $L_G(v_{s-2})\ne \emptyset$ and $\omega(v_{s-2}z) \ge \omega(v_{s-2}v_{s-1})$ for all $z\in L_G(v_{s-2})$.
\end{itemize}

For the cases (i) and (ii),  we  first prove that  $s(v,G'_\omega)=s(v,G_\omega)$. Therefore, by Claim $3$,  $k\ge s(v,G_\omega)+1$.  

This is trivial if (i)  holds. If (ii)  holds, then there exists  a leaf $r\in L_G(v_{s-1})$ such that $r\ne v_s$. By the conditions (5) and (6), $\omega(v_{s-1}r)=\omega(v_{s-1}v_{s})=\omega(v_{s-2}v_{s-1})$. In particular, $s(v,G'_{\omega})=s(v,G_{\omega})$.

For the case (iv), there exists a leaf $u\in L_G(v_{s-2})$ such that $\omega(v_{s-2}u) < \omega(v_{s-2}v_{s-1})$ and $\omega(v_{s-2}u)\leqslant \omega(v_{s-2}z)$ for all  $z\in L_G(v_{s-2})$. Together with the condition $(3)$, this yields that $\omega(v_{s-2}u)<\omega(v_{s-2}z)$ for all $z\in N_G(v_{s-2})\setminus L_G(v_{s-2})$. Therefore,  $s(v,G_\omega) = s(v, (G\setminus u)_\omega)$ and  $\omega(v_{s-2}u) \leqslant \omega(v_{s-2}z)$  for all  $z\in N_G(v_{s-2})$. 
Using the same arguments  as in   Claims $1$, $2$ and $3$, we can deduce $k\geqslant s(v,G_{\omega})+1$. 

For the cases (iii) and (v), by the condition $(3)$, we have 
\begin{equation}\label{BX1}
	\omega(v_{s-1}v_{s-2}) \leqslant \omega(v_{s-2}z) \text{ for all } z\in N_G(v_{s-2}).
	\tag{$\dag$} 
\end{equation}
Note that $s(v,G'_{\omega})=s(v,G_{\omega})-1$, and  $v_{s-1}$ is a leaf of $G'_{\omega}$  by condition $(4)$. For every $z\ne v_s$, since $\m = ((v^m,I)^k \colon f)$, $zf\in (v^m,I)^k$. Therefore,  we can write $zf$ as 
\begin{equation}\label{fz}
	fz = h g_1'\cdots g_k',
	\tag{$\ddag$} 
\end{equation}
where $h$ is a monomial and $g_1',\ldots,g_k'\in \mathcal{G}((v^m,I))$. Since  $\deg_{v_s}(zf)=\omega(v_{s-1}v_s)-1$, $g_i'\ne (v_{s-1}v_s)^{\omega(v_{s-1}v_s)}$ for all $i\in[k]$. Thus, $zf\in (v^m,I')^k$.  In particular, $\m'=((v^m,I')^k\colon f)$,
where $I'=I(G'_\omega)$. 

Substituting  $z=v_{s-1}$ into the expression  (\ref{fz}), we can obtain  that $v_{s-1}\mid g_j'$ for some $j\in[k]$, since $f\notin (v^m,I)^k$.
Note that $g_i'\ne (v_{s-1}v_s)^{\omega(v_{s-1}v_s)}$ for all $i\in [k]$, thus $g_j'=(v_{s-2}v_{s-1})^{\omega(v_{s-2}v_{s-1})}$. Therefore, $\deg_{v_{s-2}}(f)\geqslant \omega(v_{s-2}v_{s-1})$.
By Claim $2$, $v_{s-2}^{\omega(v_{s-2}v_{s-1})}v_{s-1}^{\omega(v_{s-1}v_s)}|f$. Therefore, $f$ can be written as  $f=f_1f_2$, where $f_1= v_{s-2}^{\omega(v_{s-2}v_{s-1})}\cdot v_{s-1}^{\omega(v_{s-1}v_s)}$.
Note that $v_{s-1}$ is a leaf of $G'_{\omega}$, by Lemma \ref{dec-power} and  the expression (\ref{BX1}), we have $((v^m, I')^k \colon f_1)=(v^m, I')^{k-1}$. Thus
$$
\m'= ((v^m, I')^k \colon f) = (((v^m, I')^k \colon f_1)\colon f_2) = ((v^m, I')^{k-1}\colon f_2).
$$
Therefore, $\m'\in\ass((v^m, I')^{k-1})$. By the induction hypothesis, $k-1\geqslant s(v,G'_{\omega})+1=s(v,G_{\omega})$,  implying  that $k\geqslant s(v,G_{\omega})+1$.
We complete the  proof.
\end{proof}

For a monomial $u$ in $R$, its support is $\supp(u) = \{ x_i\mid x_i \text{ divides } u \}$, i.e., it is the set of all variables appearing in $u$.
For a monomial ideal $I$ with $\mathcal{G}(I)=\{u_1,\ldots,u_m\}$, we set $\supp(I)=\bigcup\limits_{i=1}^{m}\supp(u_i)$.
Before proving the main result, we need the following  two  lemmas.

\begin{lem} \cite[Theorem 4.1]{HQ} \label{associated-primes} Let $I$ and $J$ be monomial ideals such that $\supp(I)\cap \supp(J) =\emptyset$. Then, for every $t\geqslant 1$, we have
$$\ass((I+J)^t) = \bigcup_{i=1}^t \{\p+\q\mid \p \in \ass(I^i) \text{ and } \q \in \ass(J^{t-i+1})\}.$$
\end{lem}

For a monomial ideal $I$ of $R$ and $j\in [n]$, define $I[x_j] = I R[x_j^{-1}] \cap R$ as the localization of $I$ with respect to the variable $x_j$. Note that $I[x_j] = (I : x_j^{\infty})$. More generally, for a subset $W\subseteq \{x_1,\ldots,x_n\}$, define $I[W] = IR[x^{-1}\mid x\in W] \cap R$.

\begin{lem}\label{local} If $I$ is a monomial ideal, then for all $t\geqslant 1$, we have
$$\ass(I^t) \setminus \{\m\} = \bigcup_{j=1}^n \ass(I[x_j]^t).$$
\end{lem}
\begin{proof} The proof is  similar to that of \cite[Lemma 11]{T}.
\end{proof}

We say  that  $G_\omega$ is a  trivial  tree if $|V(G_\omega)|=1$. Next, we will prove the major result of this paper.

\begin{thm}\label{main} Let $t$ be a positive integer, and let $G_\omega$ be an increasing weighted tree. If  $C$ is  a vertex cover of $G$, then  $C\in \ass (I(G_\omega)^t)$ if and only if $C$ is a strong vertex cover of $G_\omega$ and  $s(C)+1\leqslant t$.
\end{thm}
\begin{proof} Let $I=I(G_\omega)$.  According to  Lemma \ref{maxIdeal},  $\m\notin \ass(I^t)$ for all $t\geqslant 1$. Therefore,  we can assume that $C\ne V(G)$. Let $S =V(G)\setminus C$, then $S\ne \emptyset$  and $S$ is an independent set of $G$. By Lemma \ref{local}, we can deduce that $(C)\in \ass(I(G_\omega)^t)$ if and only if $(C)\in \ass(I[S]^t)$.
	
	Let $N_G(S)= \{r_1,\ldots,r_k\}$. By Lemmas  \ref{sub-tree} and   \ref{components}, we  can  assume that the connected components of $G_S$ are $T^1,\ldots, T^k, T^{k+1},\ldots, T^{\ell}$, where $r_i\in V(T^i)$ for all $i\in[k]$, and  $V(T^j)\cap N_G(S) = \emptyset$ for all $k+1\le j\le \ell$. Moreover,   $(T^i_\omega,r_i)$ and   $T^j$ are  either  trivial trees or  increasing weighted trees for all $i\in[k]$ and    $k+1\le j\le \ell$.
	
	First, we prove that 
	\begin{equation}\label{LocalS}
		I[S] = \sum_{i=1}^k (r_i^{\nu_S(r_i)},I(T^i_\omega))+\sum_{i=k+1}^{\ell} I(T^i_\omega),
		\tag{$\S$}
	\end{equation}
	where we use a convention that $I(T^i_{\omega}) = (0)$ if $T^i$ is a trivial tree.  

Indeed, $$I[S] = (x^{\nu_S(x)}\mid x\in N_G(S)) + I((G\setminus S)_\omega).$$
 For  any  $uv\in E((G\setminus S)_\omega)$, if  $u,v\in N_G(S)$, then  by Lemma \ref{components}(1),   $(uv)^{\omega(uv)} \in (x^{\nu_S(x)}\mid x\in N_G(S))$; if $u\in N_G(S)$,  $v\in C\setminus N_G(S)$ and  $\nu_S(u) \leqslant \omega(uv)$, then 
   $(uv)^{\omega(uv)} \in (x^{\nu_S(x)}\mid x\in N_G(S))$. These two facts imply that
	$$I[S] = (x^{\nu_S(x)}\mid x\in N_G(S)) + I((G_S)_\omega).$$
	Thus,
	$$I[S] = (x^{\nu_S(x)}\mid x\in N_G(S)) + \sum_{i=1}^{\ell} I(T^i_\omega)
	=\sum_{i=1}^k (r_i^{\nu_S(r_i)},I(T^i_\omega))+\sum_{i=k+1}^{\ell} I(T^i_\omega),$$
	as claimed.
	
	By Lemma \ref{associated-primes}, we can deduce that  $(C)\in \ass(I[S]^t)$ if and only if  
\[
(C)=(C_1)+\cdots +(C_\ell),
\]
where  $C_i =C\cap V(T^i)$ for all $i\in [\ell]$ such that $(C_i)\in \ass((r_i^{\nu_S(r_i)},I(T^i_\omega))^{t_i})$ for all $i\in [k]$ and
$(C_j)\in \ass(I(T^j_\omega)^{t_j})$ for all $k+1\leqslant j\leqslant \ell$. Furthermore,  $t=\sum\limits_{i=1}^{\ell} (t_i-1) + 1$ and each  $t_i\geqslant 1$.

Now, we will prove the assertion of this theorem.
	
If $(C)\in \ass(I[S]^t)$,  then, from the above description, we can see that the ideal $(C)$ can be written as an expression $(C)=(C_1)+\cdots +(C_\ell)$, where each $(C_i)$ satisfies the conditions in the above paragraph.
Note that, for  all $k+1\le j\le \ell$ and  $t_j\geqslant 1$, by Lemma \ref{maxIdeal},  $(C_j)\notin \ass(I(T^j_\omega)^{t_j})$. Therefore,  $\ell=k$.  By Lemma \ref{strong-vc},
$C$ is a strong vertex cover  of $G_\omega$. According to Lemma \ref{sG2}, we know that  for each $i\in[k]$,  $(C_i)\in \ass((r_i^{\nu_S(r_i)},I(T^i_\omega))^{t_i})$ if and only if $t_i-1\geqslant s(r_i,T^i_\omega)$. Therefore, 
\[
t=\sum_{i=1}^{k} (t_i-1) + 1 \geqslant \sum_{i=1}^{k} s(r_i,T^i_\omega)+1 = s(C)+1,
\]
where the last   equality holds by Lemma 2.9.

Conversely, if $C$ is a strong vertex cover of $G_\omega$ and $t \geqslant s(C)+1$, then,   by  Lemma \ref{strong-vc}, $s(C)=\sum\limits_{i=1}^{k} s(r_i,T^i_\omega)$.
Choose $t_i=s(r_i,T^i_\omega)+1$ for all $i\in[k-1]$ and $t_k=t-\sum\limits_{i=1}^{k-1}s(r_i,T^i_\omega)$. Then, $t_k\geqslant s(r_k,T^k_\omega)+1$ and $t=\sum\limits_{i=1}^{k}(t_i-1)+1$. By the choice of each $t_i$,
$(C_i)\in \ass((r_i^{\nu_S(r_i)}, I(T^i_\omega))^{t_i})$ by Lemma \ref{sG2}. Therefore, $(C)\in \ass(I[S]^t)$ and the proof is complete.
\end{proof}

From  the above theorem, we can derive the following two formulas.

\begin{cor} \label{Infinity} If $G_\omega$ is an increasing weighted tree, then
$$\ass^{\infty}(I(G_\omega)) = \{(C)\mid C \text{ is a strong vertex cover of $G_\omega$}\}.$$
\end{cor}

\begin{cor}\label{main-corr} If $G_\omega$ is an increasing weighted tree, then
$$\astab(I(G_\omega)) = \max\{s(C)+1\mid C \text{ is a strong vertex cover of } G_\omega\}.$$
\end{cor}

\begin{exm} Let  $G_\omega$ be a weighted path with $n\geqslant 4$  vertices and  define the weight function as follows:
$$\omega(x_ix_{i+1})=1 \text{ for any } i\in [n-2], \text{ and } \omega(x_{n-1}x_n) =2.$$
Then, $\astab(I(G_\omega)) = n-2$.
\end{exm}
\begin{proof} We can verify that the vertex cover $C = \{x_1,\ldots,x_{n-1}\}$ of $G_\omega$ is a strong vertex cover. Let $S =V(G)\setminus C$. Then 
$S = \{x_n\}$,  $G_S$ has only one connected component $T$, which is the path $x_{n-1}\to x_{n-2}\to \cdots\to x_{1}$, where $(T_\omega,x_{n-1})$ is an increasing weighted tree and $s(x_{n-1},T_\omega)=n-3$. According to  Theorem \ref{main},  $(x_1,\ldots,x_{n-1})\in \ass(I(G_\omega)^t)$ if and only if $t\geqslant n-2$.

Conversely, it is easy to show that $s(C') \leqslant n-3$ for any strong vertex cover $C'$ of $G_\omega$. According to Corollary \ref{main-corr}, $\astab(I(G_\omega)) = n-2$.
\end{proof}

\medskip
\hspace{-6mm} {\bf Data availability statement}

The data used to support the findings of this study are included within the article.

\medskip
\hspace{-6mm} {\bf Conflict of interest}

The authors declare that they have no competing interests.

\medskip
\hspace{-6mm} {\bf Acknowledgement}

 \vspace{3mm}
\hspace{-6mm}  
The third author is  supported by the Natural Science Foundation of Jiangsu Province (No. BK20221353) and the National Natural Science Foundation of China (12471246). The second author is partially supported by Vietnam National Foundation for Science and Technology Development (Grant \#101.04-2024.07). The main part of this
work was done during the second author's visit to Soochow University in Suzhou, China.  He would like to express his gratitude to Soochow University for its warm hospitality.

\end{document}